\documentclass{amsart}
\usepackage{fullpage,amssymb,algorithm,algpseudocode,graphicx}
\usepackage{hyperref}
\newcommand{\N}{\mathbb{N}}
\newcommand{\Z}{\mathbb{Z}}
\newtheorem{theorem}{Theorem}
\newtheorem{lemma}[theorem]{Lemma}
\newtheorem{corollary}[theorem]{Corollary}

\numberwithin{theorem}{section}
\numberwithin{equation}{section}
\numberwithin{figure}{section}
\numberwithin{table}{section}
\numberwithin{algorithm}{section}
\DeclareMathOperator{\lcm}{lcm}
\begin{document}
\title{Finite connected components of the aliquot graph}
\author{Andrew R. Booker}
\address{School of Mathematics, University of Bristol,
University Walk, Bristol, BS8 1TW, United Kingdom}
\email{andrew.booker@bristol.ac.uk}
\thanks{The author was partially supported by EPSRC Grant
\texttt{EP/K034383/1}.}
\begin{abstract}
Conditional on a strong form of the Goldbach conjecture,
we determine all finite connected components of the aliquot
graph containing a number less than $10^9$, as well as those containing
an amicable pair below $10^{14}$ or one of the known perfect or
sociable cycles below $10^{17}$.
Along the way we develop a fast algorithm for computing the inverse
image of an even number under the sum-of-proper-divisors function.
\end{abstract}
\maketitle
\section{Introduction}
For $n\in\N$, let $s(n)=\sum_{\substack{d\mid n\\d\ne n}}d$
denote the sum of the proper divisors of $n$. Ancient Greek mathematicians
studied the forward orbits $n$, $s(n)$, $s(s(n))$, \ldots, now called
\emph{aliquot sequences}, and noted that they sometimes enter cycles,
such as $6$, $6$, \ldots\ and $220$, $284$, $220$, \ldots.  In the modern
computer era, more than a billion examples of such \emph{aliquot cycles}
have been found \cite{chernykh,moews}; most of these, like $\{220,284\}$,
have order $2$, and are termed \emph{amicable pairs}. A long-standing
conjecture posits that there are infinitely many aliquot cycles.

One can also ask about the inverse orbits $\{n\}\cup s^{-1}(\{n\})
\cup s^{-1}(s^{-1}(\{n\}))\cup\cdots$. Although questions concerning
the inverse image $s^{-1}(\{n\})$ of a given $n$ go back at least 1000
years \cite{sesiano}, the idea of iterating the inverse map appears
to have been considered only recently (see \cite[Theorem~5.3]{egps}
and \cite{delahaye}, for instance). In relation to this, Garambois
\cite{garambois} has conducted many numerical studies, focusing in
particular on \emph{isolated cycles}, i.e.\ cycles that are their own
inverse orbits. For instance, $s^{-1}(\{28\})=\{28\}$, so $\{28\}$
is an isolated cycle of order $1$.

In this paper, we seek to generalize this concept. To do so, following
Delahaye \cite{delahaye}, we introduce the \emph{aliquot graph}, which
packages all of the aliquot sequences together into a single directed
graph. Precisely, every natural number is a node of the graph, and
for any $m,n\in\N$, there is a directed edge from $m$ to $n$ if and
only if $n=s(m)$. As the examples noted above demonstrate, the aliquot
graph is not connected; in fact any two distinct aliquot cycles lie in
distinct connected components, so presumably the graph has infinitely
many components.  In these terms, we see that Garambois' isolated cycles
are examples of finite connected components.

Our objectives are (1) to find examples of finite connected components
beyond simple cycles, and (2) to determine a comprehensive list
of all finite connected components with at least one small node.
Toward the first objective, in Section~\ref{sec:alg} we present an
algorithm for computing the inverse image $s^{-1}(\{n\})$ of a given
even number $n$ in time $O(n^{1/2+\varepsilon})$; as a corollary, we
obtain the estimate $\#s^{-1}(\{n\})\ll n^{1/2+\varepsilon}$, which
improves on a recent result of Pomerance \cite[Corollary~3.6]{pomerance}.
In Section~\ref{sec:examples} we apply the algorithm to all even amicable
pairs with smaller element below $10^{14}$, and to all known\footnote{The
list of known cycles is likely complete up to at least $10^{14}$. However,
there are many open-ended aliquot sequences beginning with a number
below that bound, so it is impossible to say for sure that the list is
complete without imposing an upper bound on the cycle length. It might
even be the case that the completeness of the list is undecidable and
cannot be certified with a finite computation!} aliquot cycles of order
$\ne2$ with smallest element below $10^{17}$. In this way we identify
many interesting examples of finite connected components.

Concerning the second objective, note first that if $p$ and $q$ are
distinct primes then $s(pq)=p+q+1$. As a slight strengthening of the
Goldbach conjecture, we have the following:
\begin{center}
\textbf{Hypothesis G}.
\textit{Every even number at least $8$ is the sum of two
distinct primes}.
\end{center}
There is ample evidence in favor of Hypothesis~G: Lu \cite{lu} showed
that it holds for all but at most $O(x^{0.879})$ even numbers $\le x$,
and Oliveira~e~Silva et al.~\cite{esilva} ran a large distributed computation
to verify it for all even $n\in[8,4\times10^{18}]$.\footnote{Strictly
speaking, they only verified the Goldbach conjecture, which is weaker
than Hypothesis~G for numbers of the form $2p$ for $p$ prime. However,
for every even $n\in[6,4\times10^{18}]$, they found a Goldbach partition
$n=p+q$ with $p\le 9781$. Hence, it suffices to verify Hypothesis~G
for $n=2p$ for all primes $p\in[5,9781]$.}

Assuming Hypothesis~G, for any odd number $n\ge9$, we have $n=p+q+1=s(pq)$
for distinct odd primes $p$, $q$. Since $pq>n$ and is again odd, we
can repeat this construction to see that the inverse orbit $\{n\}\cup
s^{-1}(\{n\})\cup s^{-1}(s^{-1}(\{n\}))\cup\cdots$ is infinite;
in particular, $n$ has infinite connected component. Note also that
$1=s(11)$, $3=s(s(9))$ and $7=s(s(49))$. Thus, under Hypothesis~G, every
odd number except $5$ has infinite inverse orbit and, since $s(5)=1$,
every odd number has infinite connected component.\footnote{Note that
$1=s(p)$ for every prime $p$, so its connected component is trivially
infinite under our definition. Some authors prefer to exclude $1$
from the aliquot graph to avoid this triviality. Fortunately, under
Hypothesis~G, the only difference that this makes to our question
of finite connected components is that $2$ and $5$ become singleton
components.} Unconditionally, Erd\H{o}s et al.~\cite[Theorem~5.3]{egps}
showed that infinite inverse orbits exist; in fact all but a density
zero subset of the odd numbers have infinite inverse orbit, although
the method of proof does not enable one to exhibit a specific such number.

We say that a connected component of the aliquot graph is
\emph{potentially infinite} if it contains an odd number. Absent a
proof of Hypothesis~G (including the Goldbach conjecture), we cannot
prove that a given potentially infinite connected component is actually
infinite, unless it is shown to contain $1$.  However, we will take
Hypothesis~G for granted in what follows, so our numerical results will
be conditional upon it.  With this caveat, in Section~\ref{sec:1e9} we
describe a computation determining the complete list of finite connected
components of the aliquot graph that contain a number below $10^9$.

Finally, in Section~\ref{sec:conclusion} we conclude with some related
questions and speculations suggested by the numerics.

\subsection*{Notation}
We shall make frequent use of the following symbols for arithmetic functions:
\begin{alignat*}{2}
\omega(n)&=\sum_{\substack{p\mid n\\p\text{ prime}}}1
&&\text{is the number of distinct prime factors of }n,\\
\Omega(n)&=\sum_{p^k\parallel n}k
&&\text{is the number of prime factors of $n$, counted with multiplicity},\\
\sigma_k(n)&=\sum_{d\mid n}d^k
&&\text{is the sum of $k$th powers of the divisors of }n,\\
\sigma(n)&=\sigma_1(n)=s&&(n)+n.
\end{alignat*}

\subsection*{Acknowledgements}
This paper was inspired by the work of Jean-Luc Garambois
\cite{garambois}, as well as posts by David Stevens and another user,
who wishes to remain anonymous, on \url{mersenneforum.org}. I thank
them for raising interesting questions. I also thank Carl Pomerance for
helpful comments and for pointing out the related results in \cite{egps}
and \cite{pomerance}.

\section{An algorithm for $s^{-1}$}\label{sec:alg}
Suppose that $n\in\N$ is given, and we wish to find $m\in\N$ satisfying
$s(m)=n$. If $n\ge9$ is odd, then searching through small primes $p$,
we expect to find one quickly (polynomial time in $\log{n}$) such
that $q=n-1-p$ is prime, so that $n=s(pq)$; although a proof of this
seems far off, that does not prevent it from working well in practice
to find an element of $s^{-1}(\{n\})$, even for very large odd $n$.
On the other hand, it is conjectured that all large odd $n$ have $\gg
n/\log^2n$ representations as $p+q+1$ (and this certainly holds for
at least some arbitrarily large $n$, by the prime number theorem and
pigeonhole principle), and it follows that no algorithm can compute all
of $s^{-1}(\{n\})$ in fewer than $O(n/\log{n})$ bit operations.
In light of this, and since our application requires only even values,
we assume henceforth that $n$ is even.

Let us first consider the possibility of odd $m$.
If $n\in2\N$ and $m\in1+2\N$, then it is easy to see
that $m$ must be a square. Let $p$ be the largest prime factor of $m$,
and write $m=a^2p^{2k}$, with $p\nmid a$. Then we have
\begin{equation}\label{eq:sa2}
n=s(m)=s(a^2)p^{2k}+\sigma(a^2)(1+p+\ldots+p^{2k-1}),
\end{equation}
so that $a^2\le n/(1+\ldots+p^{2k-1})$ and
$k\le\frac12[1+\log_p(n/\sigma(a^2))]$.
For $a=1$ and each odd $a\in[3,\sqrt{n/6}]$, we run through all
$k\le\frac12[1+\log_q(n/\sigma(a^2))]$, where $q$ is the smallest odd
number $\ge3$ exceeding every prime factor of $a$, perform a binary search for
integral $p\ge q$ satisfying \eqref{eq:sa2}, and apply a primality test.
(For our implementation, which was limited to $n<2^{64}$, we used a
strong Fermat test to base $2$, together with the classification \cite{feitsma}
of small strong pseudoprimes.)

Next we consider $m\in2\N$. In this case, since $m/2$ is a proper divisor of
$m$, we have $s(m)\ge m/2$, whence $m\le2n$. We write
$m$ in the form $ab$, where we
think of $a\in2\N$ as the ``smooth'' part of $m$, with only small prime
factors, and $b$ as the rest. Then we have
\begin{equation}\label{eq:ab}
n=s(m)=\sigma(a)s(b)+s(a)b.
\end{equation}
For a fixed choice of $a$, we
view \eqref{eq:ab} as a linear equation constraining the pair
$(s(b),b)$. First note that if $g=\gcd(\sigma(a),s(a))=\gcd(a,s(a))$, then
\eqref{eq:ab} has no solutions unless $g\mid n$. When $g\mid n$, we put
$u=\sigma(a)/g$ and $v=s(a)/g$, so that
$(x,y)=(s(b),b)$ is a solution to $ux+vy=n/g$.
Using the Euclidean algorithm, we can determine a particular solution
$(x_0,y_0)\in\Z^2$, and the general solution in positive integers is
then $(x,y)=(x_0+rv,y_0-ru)$ for $r\in\Z\cap(-\frac{x_0}v,\frac{y_0}u)$.
Our algorithm proceeds by working recursively through all possible
prime factorizations of $a$. For the base case of the recursion, once
the number of possibilities for $b$ is small enough, we test all of
them to see if the equality $n=s(ab)$ is satisfied.

As described, this method is only a little more
efficient than directly considering every possible even $m\le 2n$, but
fortunately there are a few ways in which we can reduce the search
space. First, we can detect the cases $b=p$ or $b=p^2$ for a prime $p$
by solving \eqref{eq:ab}, which gives a linear or quadratic equation for
$p$, and applying a primality test.
Second, in the typical case when $b$ has no small prime factors, we
can narrow the range for $s(b)$ using the following estimate:
\begin{lemma}
Let $b>1$ be an integer with smallest prime factor $p$. Then
$s(b)\in[b/p,b\Omega(b)/p]$.
\end{lemma}
\begin{proof}
Since $b/p$ is a proper divisor of $b$, we have $s(b)\ge b/p$,
directly from the definition. For the upper bound,
let $\prod_{i=1}^{\omega(b)}p_i^{e_i}$ be the
prime factorization of $b$, consider the set
$$
S=\bigcup_{i=1}^{\omega(b)}\{p_i,p_i^2,\ldots,p_i^{e_i}\},
$$
and write $S=\{q_1,\ldots,q_{\Omega(b)}\}$, with
$q_1<\cdots<q_{\Omega(b)}$ in increasing order.
Next set $b_0=1$ and $b_j=\lcm(q_1,\ldots,q_j)$ for
$j=1,\ldots,\Omega(b)$. Then
$$
\sigma_{-1}(b)=\prod_{j=1}^{\Omega(b)}
\frac{\sigma_{-1}(b_j)}{\sigma_{-1}(b_{j-1})}.
$$
Consider $j\in\{1,\ldots,\Omega(b)\}$, and suppose that
$q_j=p_i^k$. Then $b_j=p_ib_{j-1}$ and
$$
\frac{\sigma_{-1}(b_j)}{\sigma_{-1}(b_{j-1})}
=\frac{\sigma_{-1}(p_i^k)}{\sigma_{-1}(p_i^{k-1})}
=1+\frac1{p_i+\ldots+p_i^k}
\le1+\frac1{q_j}.
$$
Since $q_1=p$ and the $q_j$ are strictly increasing, we thus have
$$
\sigma_{-1}(b)\le\prod_{j=1}^{\Omega(b)}\left(1+\frac1{q_j}\right)
\le\prod_{j=1}^{\Omega(b)}\left(1+\frac1{p+j-1}\right)
=1+\frac{\Omega(b)}{p}.
$$
Hence
$$
\frac{s(b)}{b}=\sigma_{-1}(b)-1\le\frac{\Omega(b)}{p}.
$$
\end{proof}
Although we do not know
$p$ in advance, we will know a lower bound for it in the course of the
recursion. Supposing that $p\ge p_1$ and that we have already checked
the cases $b=1$, $b=p$ and $b=p^2$, we have
\begin{equation}\label{eq:b1s1}
b\ge b_1:=p_1p_1'\quad\text{and}\quad s(b)\ge s_1:=1+p_1+p_1',
\end{equation}
where $p_1'$ denotes the smallest prime exceeding $p_1$. Thus, defining
\begin{equation}\label{eq:b2s2}
b_2=\frac{n-\sigma(a)s_1}{s(a)},\quad
k=\left\lfloor\frac{\log{b_2}}{\log{p_1}}\right\rfloor
\quad\text{and}\quad
s_2=\frac{kn}{k\sigma(a)+p_1s(a)},
\end{equation}
we have
$$
s(b)(\sigma(a)+s(a)p_1/k)\le\sigma(a)s(b)+s(a)b=n,
$$
so that $b\in[b_1,b_2]$ and $s(b)\in[s_1,s_2]$.
We stop the recursion and test every value of $b$ once the number of
$(x,y)\in[s_1,s_2]\times[b_1,b_2]$ satisfying
$ux+vy=n/g$ falls below $p_1$.

Third, the solutions with $b=pq$ for distinct primes $p$ and $q$ can
also be determined without searching, since in this case we have
$$
(vp+u)(vq+u)=v^2pq+uv(p+q)+u^2=v(n/g-u)+u^2=(au+nv)/g.
$$
Thus, factoring $(au+nv)/g$ and testing all of its divisors
$\equiv u\pmod{v}$ will reveal $p$ and $q$. Since $(au+nv)/g$ is
potentially quite large, this test is more expensive than that for $b=p$
or $p^2$, so we use it only when $p_1$ is large enough to guarantee that
$b$ is a product of two primes.

\algnewcommand\algorithmicinput{\textbf{Input:}}
\algnewcommand\Input{\item[\algorithmicinput]}
\algnewcommand\algorithmicoutput{\textbf{Output:}}
\algnewcommand\Output{\item[\algorithmicoutput]}
\algnewcommand{\IIf}[1]{\State\algorithmicif\ #1\ \algorithmicthen}
\algnewcommand{\EndIIf}{\unskip\ \algorithmicend\ \algorithmicif}
\begin{algorithm}
\caption{Procedure to compute $s^{-1}(\{n\})$ for $n\in2\N$}\label{alg}
\begin{algorithmic}
\Function{s\_inverse}{$n$}
\Input{$n\in2\N$}
\Output{list of $m\in\N$ such that $s(m)=n$}
\State initialize the output list
\For{each $a\in\{1\}\cup[3,\sqrt{n/6}]\cap(1+2\N)$}
\State compute $s(a^2)$ and the smallest odd number $q\ge3$ exceeding
every prime factor of $a$
\For{each $k\in\N$ such that $q^{2k-1}\le n/\sigma(a^2)$}
\State solve \eqref{eq:sa2} for $p$
\IIf{$p$ is a prime $\ge q$} append $a^2p^{2k}$ to the output list \EndIIf
\EndFor
\EndFor
\For{each $k\in\N$ such that $2^k<n$}
\State call \Call{s\_inverse\_even\_recursion}{$2^k$}
\EndFor
\State\Return the output list
\EndFunction

\medskip
\Procedure{s\_inverse\_even\_recursion}{$a$}
\Input{$a\in 2\N$}
\Ensure{appends to the output list all $m=ab$ such that $s(m)=n$, $b>1$ and
the smallest prime factor of $b$ exceeds the largest prime factor of $a$}
\State compute $g=\gcd(s(a),\sigma(a))$, and \Return if $g\nmid n$
\State check for solutions to \eqref{eq:ab} with $b=p$ and $b=p^2$, and
append them to the output list
\State compute $u=\sigma(a)/g$, $v=s(a)/g$, and $(x_0,y_0)$ such that
$ux_0+vy_0=n/g$
\For{primes $p_1$ exceeding the largest prime factor of $a$, in
increasing order,}
\State compute the intervals $[s_1,s_2]$ and $[b_1,b_2]$ defined in
\eqref{eq:b1s1}--\eqref{eq:b2s2}
\If{$\#\{r\in\Z:x_0+rv\in[s_1,s_2]\text{ and }y_0-ru\in[b_1,b_2]\}<p_1$}
\For{each such $r$}
\State compute $b=y_0-ru$ and $s(b)$
\If{every prime factor of $b$ is at least $p_1$ and $s(b)=x_0+rv$}
\State append $ab$ to the output list
\EndIf
\EndFor
\State\Return
\EndIf
\If{$s(ap_1^3)>n$}
\State factor $N=(au+nv)/g$ and find all of its divisors $d<\sqrt{N}$
satisfying $d\equiv u\pmod{v}$
\For{each such $d$}
\State compute $p=(d-u)/v$ and $q=(N/d-u)/v$
\IIf{$p$ and $q$ are primes $\ge p_1$}
append $apq$ to the output list
\EndIIf
\EndFor
\State\Return
\EndIf
\For{each $k\in\N$ such that $s(ap_1^k)\le n$}
\If{$s(ap_1^k)<n$}
\State call \Call{s\_inverse\_even\_recursion}{$ap_1^k$}
\ElsIf{$k\ge3$}
\State append $ap_1^k$ to the output list
\EndIf
\EndFor
\EndFor
\EndProcedure
\end{algorithmic}
\end{algorithm}

Our procedure is described in more detailed pseudocode in
Algorithm~\ref{alg}. We turn now to the running time analysis. First,
by either using a sieve to amortize the factorization of $a$ or working
recursively through the possible factorizations,
we see that it takes at most $O_\varepsilon(n^{1/2+\varepsilon})$
bit operations to find all odd $m$ with $s(m)=n$.
For even $m$, note that each prime $p_1$ considered before the recursion
is stopped satisfies
$$
p_1\le\#\{r\in\Z:x_0+rv\in[s_1,s_2]\text{ and }y_0-ru\in[b_1,b_2]\}
\le\frac{s_2}{v}+1,
$$
and together with \eqref{eq:b2s2} this implies the bound
$p_1\le\frac{\sqrt{gn\log_3{n}}}{s(a)}$.
To simplify the analysis, we
consider a modified version of the algorithm in which
we omit the checks for $b=p$, $b=p^2$ and $b=pq$, and stop the recursion
once $p_1>\sqrt{n}/a$. (These simplifications make the algorithm slightly
less efficient, but one can see that they increase the running time by a
factor of $O_\varepsilon(n^\varepsilon)$ at most.)

Suppose that the recursive procedure
is called with input $a$, and let $p$ denote the largest prime factor
of $a$, with $p^k\parallel a$. Then either $p=2$ or the criterion for
stopping the recursion was not satisfied when considering $a/p^k$, so
that $p\le\sqrt{n}/(a/p^k)$. We may assume that $n\ge4$, so in either
case, writing $f(a)=a/p^{k-1}$, we have $f(a)\le\sqrt{n}$.
Note that $f(a)$ is again an even integer with largest prime factor $p$.
Thus,
\begin{align*}
\#\{a\in2\N:a\le2n,\,f(a)\le\sqrt{n}\}
&=\sum_{\substack{t\in2\N\\t\le\sqrt{n}}}\#\{a\in2\N:a\le2n,\,f(a)=t\}\\
&\le\sum_{\substack{t\in2\N\\t\le\sqrt{n}}}
\left(1+\log_2\bigl(\tfrac{2n}{t}\right)\bigr)
\le\tfrac12\sqrt{n}\log_2(2n),
\end{align*}
and this gives an upper bound for the number of times
that the recursive procedure is called.

Next, let $p_1$ denote the smallest prime exceeding both $\sqrt{n}/a$
and every prime factor of $a$. Then by \eqref{eq:b2s2}, the values of
$b$ that we consider in the base case of the recursion for $a$ satisfy
$$
s(b)<\frac{kn}{p_1s(a)}\le\frac{n\log_{p_1}{n}}{p_1s(a)}
\le\frac{n\log_3{n}}{(\sqrt{n}/a)(a/2)}
=2\sqrt{n}\log_3{n}.
$$
Moreover, $s(b)$ is determined modulo $v=s(a)/g$, so the number of
possibilities to consider is at most
$$
1+\frac{2\sqrt{n}\log_3{n}}{s(a)/g}
\le1+\frac{4g\sqrt{n}\log_3{n}}{a}.
$$
Summing over all $g\mid n$ and $a$ satisfying $\gcd(s(a),\sigma(a))=g$,
we see that the total number of candidate values for $b$ is bounded by
\begin{align*}
\sum_{g\mid n}\sum_{\substack{a\in2\N\cap[2,2n]\\
f(a)\le\sqrt{n}\\\gcd(a,\sigma(a))=g}}
\left(1+\frac{4g\sqrt{n}\log_3{n}}{a}\right)
&\le\sum_{\substack{a\in2\N\cap[2,2n]\\f(a)\le\sqrt{n}}}1
+\sum_{g\mid n}\sum_{\substack{a\le 2n\\g\mid a}}
\frac{4\sqrt{n}\log_3{n}}{a/g}\\
&\ll\sigma_0(n)\sqrt{n}\log^2{n}\ll_\varepsilon n^{1/2+\varepsilon}.
\end{align*}

The largest prime factor of $a$ and the value of $s(a)$ can be carried
along as extra state information during the recursion, so no work
is required to factor $a$. On the other hand, we can expect the $b$
values that arise to occur sparsely throughout $(0,n)$, and we need
to factor them in order to compute $s(b)$. In practice, one can use a
generic factoring algorithm with good average-case performance. To get
a provable estimate for the running time, it suffices to record all
of the candidate pairs $(a,b)$ in a list and apply Bernstein's batch
factorization algorithm \cite{bernstein} to the $b$ values. Since
there are $O_\varepsilon(n^{1/2+\varepsilon})$ pairs and each $b$
is bounded by $n$, the total time to factor all of them is still
$O_\varepsilon(n^{1/2+\varepsilon})$.

Thus, we have shown the following.
\begin{theorem}
The algorithm described in this section computes $s^{-1}(\{n\})$ for a
given $n\in2\N$ in time at most $O_\varepsilon(n^{1/2+\varepsilon})$.
\end{theorem}
\begin{corollary}
For $n\in2\N$, $\#s^{-1}(\{n\})\ll_\varepsilon n^{1/2+\varepsilon}$.
\end{corollary}

\section{Numerical results}
\subsection{Examples of finite connected components}\label{sec:examples}
For any given $n\in\N$, the forward orbit of $n$ under $s$ either
terminates with $1$, grows without bound, or enters a cycle. In the
first two cases, $n$ must have infinite connected component. Hence,
to find finite connected components, it suffices to consider only those
$n$ contained in a cycle, and compute their inverse orbits.
For each even amicable pair with smaller element below $10^{14}$, as
well as the known perfect or sociable cycles with smallest element below
$10^{17}$, we started with the smallest $n$ in the cycle and
iteratively computed $s^{-1}(\{n\})$, $s^{-1}(s^{-1}(\{n\}))$, \ldots\
until reaching either the empty set or a set containing an odd number.
In the former case, $n$ has finite connected component, and our
computation determines it entirely; in the latter case, assuming
Hypothesis~G, the connected component is infinite.

It is also conceivable that there are $n$ for which neither case occurs,
and the procedure does not terminate. However, for any $n$, the elements
of $s^{-1}(\{n\})\cap2\N$ are bounded by $2n$, so chains of even numbers
in the inverse orbit of $n$ grow at most exponentially in the iteration
count. Moreover, for any $m$ of the form $p+1$ for prime $p$, we have
$m=s(p^2)$. We see no reason why numbers of this form should not occur
among the elements of the inverse orbit of $n$ with the same frequency
as for random numbers of the same size. Thus, provided that the $k$th
iterate of $s^{-1}$ applied to $\{n\}$ is non-empty, we expect it to
contain an odd number with probability $\gg 1/k$. Since the harmonic
series diverges, we therefore expect to reach an odd number eventually,
as long as the inverse orbit is infinite. This was borne out by our
numerics, as every connected component that we considered was found to
be either finite or potentially infinite.\footnote{However, in the case
of the amicable pair $\{48569114359984,49074636040016\}$, the numbers
exceeded the 64-bit limit of our implementation without reaching an odd
number. We wrote a special-purpose routine to continue the search in
this case, looking for $m$ of the form $ap$ with $a<2\times10^{10}$
and $p$ prime, and fortunately that sufficed to prove that
$s^{25}(18471983707171354573^2)=49074636040016$.}

Of the $24003$ even amicable pairs that we considered,
$7438$ pairs were found to belong to a finite connected component, and
of those, $2394$ were isolated cycles. The average size of the
components was $37968/7438\approx 5.1$, and the largest was of size
$58$, corresponding to the amicable pair
$\{29215166389256,31021462090744\}$; it is shown in
Figure~\ref{fig:amicable}.
\begin{figure}
\caption{The largest finite connected component containing an amicable
pair with smaller element below $10^{14}$}\label{fig:amicable}
\begin{center}
\includegraphics[width=\textwidth]{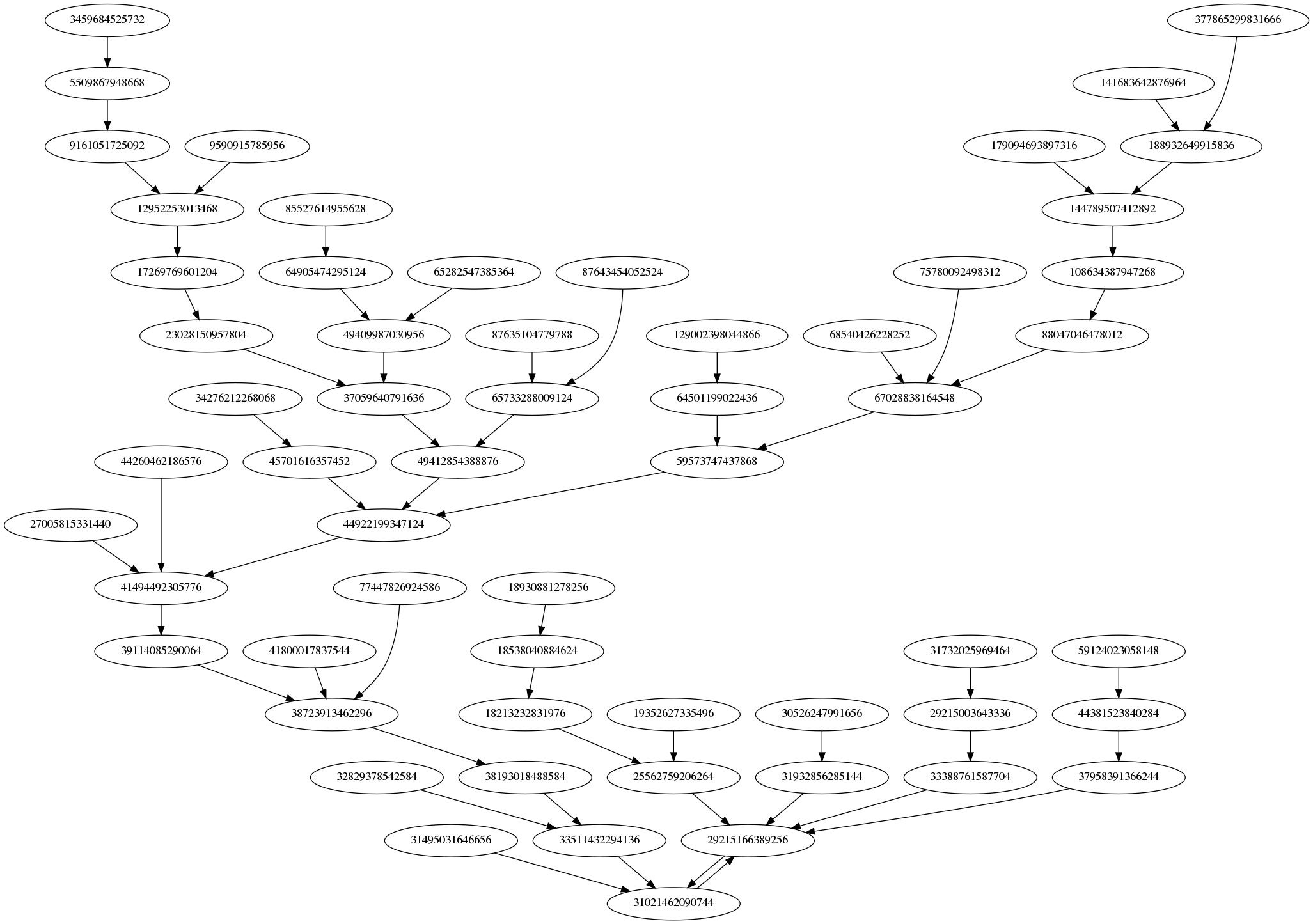}
\end{center}
\end{figure}

For even aliquot cycles of size other than $2$,
only $75$ are known with smallest
element below $10^{17}$. We found $12$ belonging
to a finite connected component, of which three are isolated cycles
(including the perfect numbers $28$ and $137438691328$);
they are shown in Figure~\ref{fig:sociable}.

\begin{figure}
\caption{The finite connected components
containing a known cycle of order $\ne2$ with a node
$\le10^{17}$}\label{fig:sociable}
\begin{center}
\includegraphics[width=\textwidth]{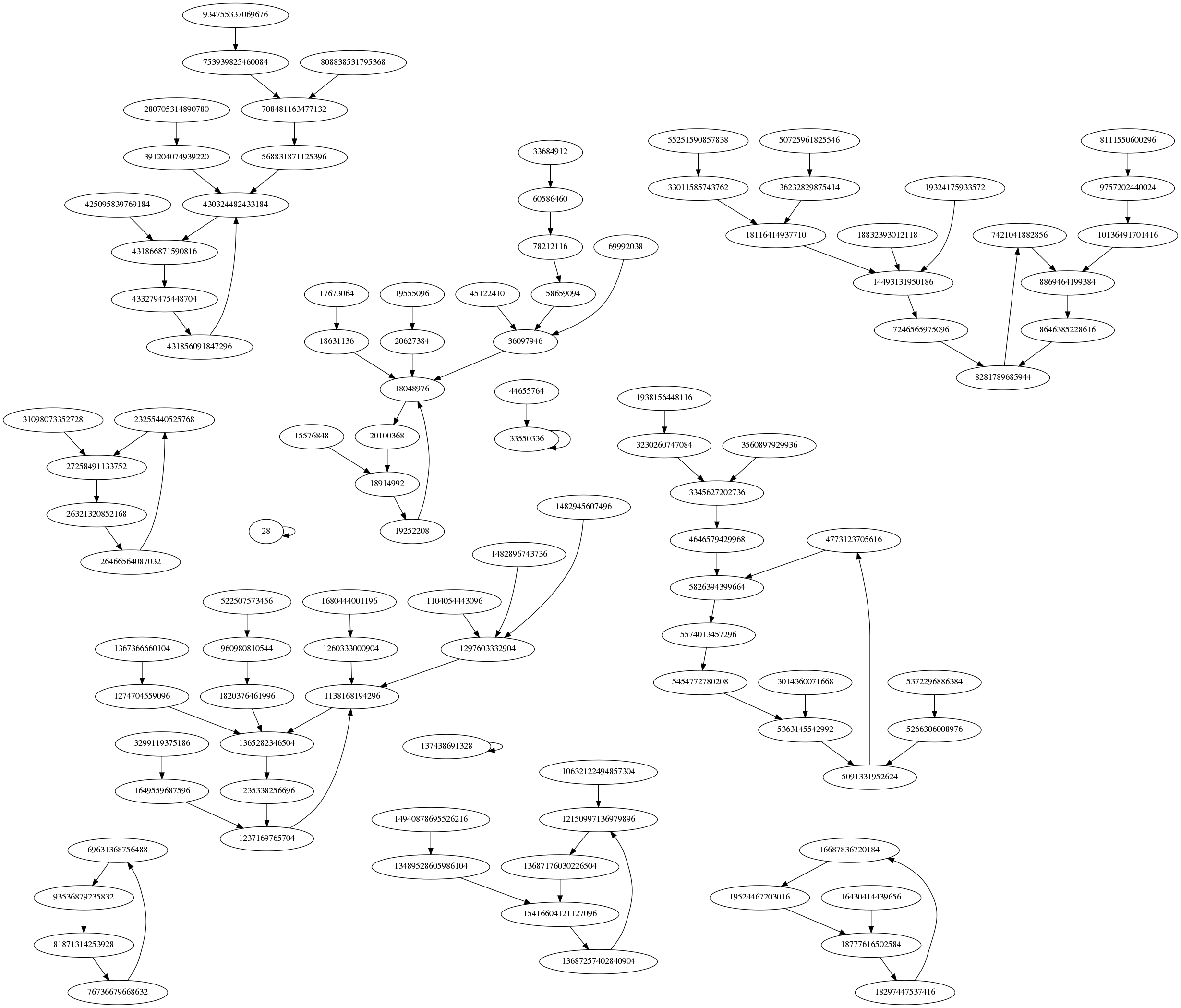}
\end{center}
\end{figure}

\subsection{Finite connected components containing a small node}\label{sec:1e9}
Towards our second objective, we found, conditional on Hypothesis~G, the
complete list of finite connected components of the aliquot graph
containing a number $\le10^9$. As it turns out, there are $101$ such
components, compared to $453$ known cycles of even numbers in that range.
They are comprised of $462$ nodes, $88$ of which exceed $10^9$.
The $14$ examples containing a number below $10^7$ are
shown in Figure~\ref{fig:1e7}.

\begin{figure}
\caption{All finite connected components
containing a node $\le10^7$}\label{fig:1e7}
\begin{center}
\includegraphics[width=0.6\textwidth]{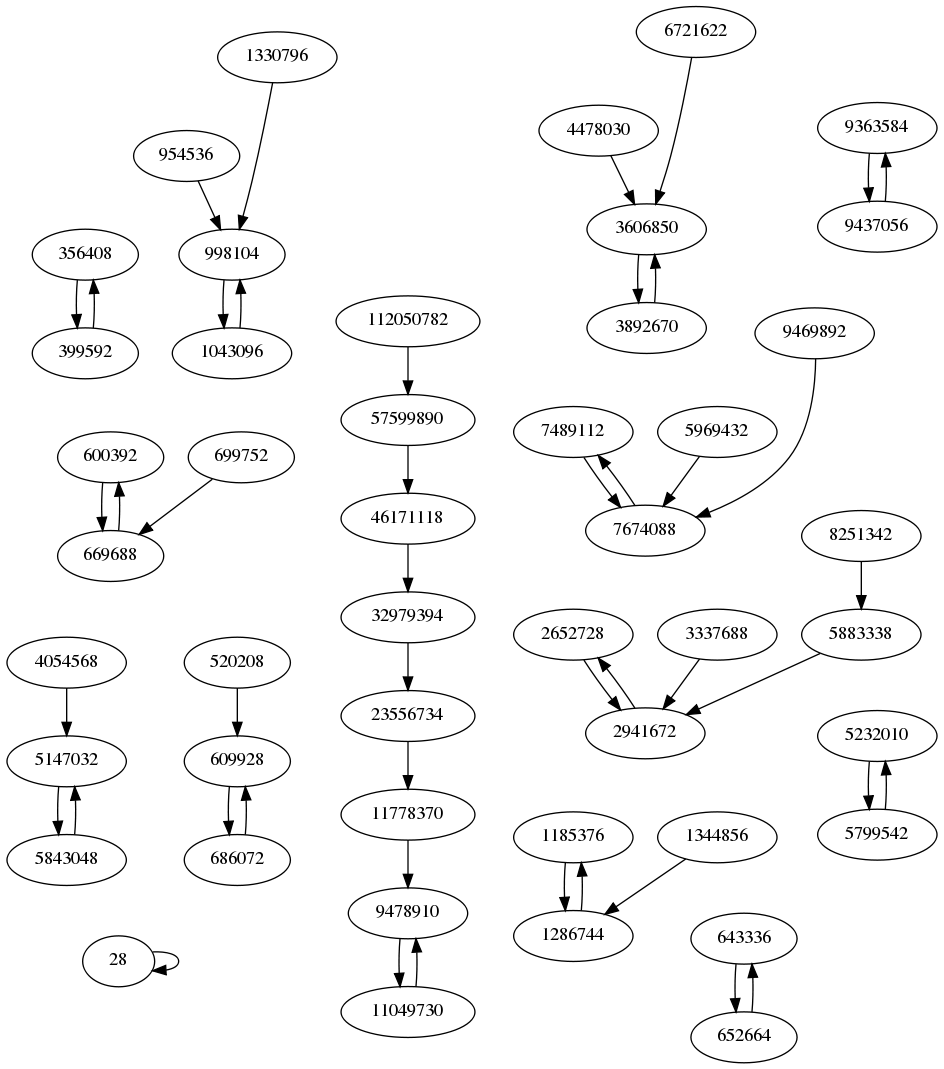}
\end{center}
\end{figure}

Our computation proceeded as follows.
First, beginning with each even number $n\le10^9$, we used
PARI/GP \cite{pari} to compute the forward orbit
$n$, $s(n)$, \ldots, until arriving at a number $m=s^k(n)$
satisfying one of the following conditions:
\begin{enumerate}
\item $m$ is odd;
\item $m-1$ is prime;
\item $m=s^j(n)$ for some $j<k$;
\item $m\ge 10^{50}$.
\end{enumerate}
In the first two cases, $n$ is connected to an odd number ($m$
in case (1), $(m-1)^2$ in case (2)), so its connected component is
potentially infinite. In the third case, the forward orbit enters a
cycle. We determined the minimum number in each cycle and collated
the cycles discovered for all $n\le10^9$. It turned out that they were
all among the cycles considered in Section~\ref{sec:examples}, so we
could readily classify each connected component as either finite or
potentially infinite.  That left $1053$ numbers in the indeterminate case
(4), to which we applied the algorithm from Section~2 to search for odd
numbers in the inverse orbit of $n$, then of $s(n)$, $s(s(n))$, \ldots,
until reaching a value of $s^k(n)$ in excess of $2^{48}$.  With this method
we succeeded in finding an odd number for all but nine values of $n$, whose
forward orbits merged into just four distinct aliquot sequences. Finally,
we resolved these by continuing the forward orbits with a larger cutoff
of $10^{70}$.

For each $n$ with potentially infinite connected component, we recorded,
as a certificate, an odd number $m\in\N$ and indices $j,k\ge0$ such that
$s^k(n)=s^j(m)$. The interested reader may find these at \cite{data},
along with the data pertaining to the finite connected components.

\section{Related questions}\label{sec:conclusion}
Recall that a number $n\in\N$ is called \emph{non-aliquot} (or
\emph{untouchable}) if $s^{-1}(\{n\})=\emptyset$.
Pollack and Pomerance \cite{pp} have conjectured that the non-aliquot
numbers have asymptotic density
$$
\lim_{y\to\infty}
\frac{\sum_{\substack{a\in2\N\\a\le y}}a^{-1}e^{-a/s(a)}}
{\sum_{\substack{a\in\N\\a\le y}}a^{-1}}
\approx17\%
$$
in the natural numbers, and this is supported by the available numerical
evidence. Their analysis relies heavily on some heuristics for the typical
behavior of $s$ over the natural numbers. The heuristics
do not apply to amicable numbers, which are atypical in this respect
(e.g., for any amicable number $a$, the sequence $a$, $s(a)$,
$s(s(a))$ is not monotonic, which is a rare event among all natural
numbers). Nevertheless, one can ask whether the amicable pairs
that form isolated cycles have a density within the set of all amicable
pairs (ordered by smaller element, say). Empirically almost all aliquot
cycles have order $2$, so this density, if it exists, should agree with
that of the isolated cycles among all cycles. Table~\ref{tab:density}
shows the frequency of isolated cycles among all known cycles in various
ranges up to $10^{14}$.  Based on this limited evidence, we speculate
that the limiting density does exist and is approximately $6\%$.
\begin{table}
\caption{Frequency of isolated cycles and cycles with finite connected
component}\label{tab:density}
\begin{tabular}{rrrr}
    & number of cycles with    & number that  & number with finite \\
$x$ & smallest element $\le x$ & are isolated & connected component \\ \hline
$10^{10}$ & $1462$ & $98$ ($6.70\%$) & $249$ ($17.03\%$)\\
$10^{11}$ & $3385$ & $214$ ($6.32\%$) & $613$ ($18.11\%$)\\
$10^{12}$ & $7692$ & $471$ ($6.12\%$) & $1445$ ($18.79\%$)\\
$10^{13}$ & $17583$ & $1052$ ($5.98\%$) & $3309$ ($18.82\%$)\\
$10^{14}$ & $39457$ & $2397$ ($6.07\%$) & $7448$ ($18.88\%$)
\end{tabular}
\end{table}

Similarly, one might ask whether there are infinitely many finite
connected components, and whether the cycles with finite connected
component have a density among all cycles. Table~\ref{tab:density}
also shows data relevant to these questions. Again we speculate that the
answer to both is yes, with the limiting density approximately $19\%$.

Finally, we found finite connected components of every size $\le41$.
Table~\ref{tab:large} shows the ones of record size when ordered by
smallest element.  It seems plausible that every positive integer is the
cardinality of a finite connected component; in particular, we conjecture
that there are arbitrarily large finite components.
\begin{table}
\caption{Numbers with finite connected component
of record size}\label{tab:large}
\begin{tabular}{rr|rr}
$n$ & size & $n$ & size\\ \hline
$28$ & $1$ &
$7651954416$ & $24$\\
$356408$ & $2$ &
$10238969536$ & $35$\\
$520208$ & $3$ &
$97624271600$ & $36$\\
$954536$ & $4$ &
$757688279778$ & $37$\\
$2652728$ & $5$ &
$944013126176$ & $38$\\
$9478910$ & $8$ &
$1164087362100$ & $41$\\
$15576848$ & $16$ &
$1336635061736$ & $52$\\
$932913124$ & $21$ &
$3459684525732$ & $58$
\end{tabular}
\end{table}
\bibliographystyle{amsplain}
\bibliography{aliquot}
\end{document}